\documentclass[12pt,draft]{amsart}
\usepackage{amsmath,amstext,amsfonts,amssymb,graphicx,framed,dsfont,enumitem,xcolor,tikz,comment}
\usepackage{mathrsfs,lineno}
\usepackage{float}
\DeclareSymbolFontAlphabet{\mathrsfs}{rsfs}
\usepackage[mathscr]{eucal}

\usepackage{xcolor}

\newtheorem{prop}{Proposition}[section]

\newtheorem{lem}[prop]{Lemma}
\newtheorem{cor}[prop]{Corollary}
\newtheorem{thm}[prop]{Theorem}
\newtheorem{conjecture}[prop]{Conjecture}

\theoremstyle{definition}

\newtheorem{rem}[prop]{Remark}
\newtheorem{rems}[prop]{Remarks}

\newcommand\Hy {\mathcal{H}}
\newcommand\de{\mathop{\mathrm{deg}_\Hy}}

\begin{document}

\title{The Berge-F\"uredi conjecture on the \\
chromatic index of hypergraphs with large hyperedges
}
\author{Alain Bretto}
\address{NormandieUnicaen, GREYC CNRS-UMR 6072, Caen, France}
\email{alain.bretto@info.unicaen.fr}
\date{\today}


\author{Alain Faisant}
\author{François Hennecart}
\address{
Universit\'e Jean Monnet, ICJ UMR5208, CNRS, Ecole Centrale de Lyon, INSA Lyon, Universite Claude Bernard Lyon 1,
42023 Saint-Etienne, France}
\email{faisant@univ-st-etienne.fr, francois.hennecart@univ-st-etienne.fr}

\begin{abstract}

This paper is concerned with two conjectures which are intimately related. The first is a generalization to hypergraphs of  Vizing's Theorem on the chromatic index of a graph and the second is the  well-known conjecture of Erd\H{o}s, Faber and Lov\'asz which deals with  the problem of coloring a family of cliques intersecting in at most one vertex. We are led to study a special class of uniform and linear hypergraphs for which a number of properties are established.

\end{abstract}

\maketitle

\section{\bf Introduction}
\label{intro}
In this paper, we address the problem of coloring hyperedges of a hypergraph $\Hy$, in the case where $\Hy$ is loopless and linear \cite{Berge1997,Berge1989}. Two main conjectures have been proposed.
The first one (the last chronologically) is a conjecture of several authors which is a generalization of  Vizing's Theorem on the coloration of edges in a graph (cf. \cite{Viz1964}).
We postpone to Section \ref{S2} all the required definitions and notations.

\begin{thm}[\textsc{Vizing's Theorem}, 1964]
\label{theo11}
The chromatic index of any simple graph~$\Gamma$ satisfies
\[
\mathrm{q}(\Gamma)\leq \Delta(\Gamma)+1.
\]
\end{thm}

One can wonder if this result can be generalized to hypergraphs.  Actually, this theorem suggested the following conjecture.

\begin{conjecture}[\textsc{Generalized Vizing’s Theorem}]\label{conj12}
Every linear hypergraph $\Hy$ without  loop verifies 
\begin{equation}\label{eqqH}
\mathrm{q}(\Hy)\leq \Delta([\Hy]_{2})+1,
\end{equation}
where $\Delta([\Hy]_{2})$ is the maximum degree of the $2$-section $[\Hy]_{2}$ of $\Hy$.
\end{conjecture}

This conjecture was independently proposed, around 1985 by several authors, including Berge (cf. \cite{Berge1989}), F\"uredi (cf. \cite{Furedi1986} where the conjecture is proven for intersecting hypergraphs) and Meyniel (unpublished).

The second one is closely related to the first and deals on the coloring of vertices of a graph formed by a families of $n$ cliques each having $n$ vertices and intersecting in  at most one vertex.

 \begin{conjecture}[\textsc{Erd\H{o}s-Faber-Lov\'asz}]
\label{conj13}
Every graph $\Gamma$ formed by a family of $n$ cliques each having $n$ vertices and intersecting in at most one vertex has chromatic number verifying
$\chi(\Gamma)\leq  n$.
\end{conjecture}

This conjecture can be implemented in hypergraph's language by considering  cliques as hyperedges. Transposing it to the dual gives the following statement.

\begin{conjecture}[\textsc{Erd\H{o}s-Faber-Lov\'asz}]\label{conj14}
Every linear hypergraph $\Hy$ without loop and having $n$ vertices verifies
$\mathrm{q}(\Hy)\leq n$.
\end{conjecture}

It is easy to see that Conjecture \ref{conj12} implies Conjecture \ref{conj14} since $\Delta([\Hy]_{2})\leq n-1$ when $\Hy$ is linear.
There are very few results on these conjectures, especially on Conjecture \ref{conj12} (see \cite{Furedi1986, Berge1989}).  

In this article we solve the  Berge-F\"uredi Conjecture~\ref{conj12} whenever the antirank is bounded from below by the square root of the number of vertices.
We first provide a criterion ensuring that $\Hy$  satisfies  
the desired bound \eqref{eqqH}
and use it to show our main result in regards to Conjecture~\ref{conj12}  (cf. Theorem \ref{theo16}).

\begin{thm}\label{theo15}
Let $\Hy =(V,E)$ be a linear hypergraph such that
\begin{equation}\label{cond*}
\forall e\in E,\ \exists x\in e\text{ such that }\mathrm{deg}_{\Hy}(x)\le |e|. 
\end{equation}
Then 
\begin{equation}\label{eqHle}
\mathrm{q}(\Hy)\le \Delta([\Hy]_2)+1.
 \end{equation}
\end{thm}

Note that if $\Delta(\Hy)\le \mathrm{ar}(\Hy)$ then \eqref{cond*} is satisfied, whence $\mathrm{q}(\Hy)\le \Delta([\Hy]_2)+1$ by Theorem \ref{theo15}. Thus the validity of Conjecture  \ref{conj12} depends actually on the remaining case 
$\Delta(\Hy)>\mathrm{ar}(\Hy)$.

In  \cite{Sanchez2008} the author shows that
$\mathrm{q}(\Hy)\le |V|$ whenever $\mathrm{ar}(\Hy) > |V|^{1/2}$.
In \cite{Wang2024} it is proven that it remains true under the weaker condition
 $\mathrm{ar}(\Hy) \ge  |V|^{1/2}$
We improve this result by showing the following statement.

\begin{thm} \label{theo16}
Let $\Hy =(V,E)$ be a linear hypergraph such that 
$\mathrm{ar}(\Hy)\ge |V|^{1/2}$. 
\\
Then 
\[
\mathrm{q}(\Hy)\le \Delta([\Hy]_2)+1.
\]
\end{thm}

{The special case in Theorem \ref{theo16} where $\mathrm{ar}(\Hy)=|V|^{1/2}$ highlights an attractive class, denoted $\mathrsfs{H}_k$, of linear and $k$-uniform hypergraphs which will be studied separately (see beginning of Section \ref{S4} for detailed definition). We will obtain various remarkable properties  in Theorem \ref{theo36} and Section \ref{S4} (see Theorems \ref{theo43} and \ref{theo44}) which will raise up several questions.
This is the core of this paper.}

\section{\bf Definitions, notations and general properties}
\label{S2}

Throughout this article, we shall use the notation $|S|$ for the cardinality of a finite set $S$. In what follows $\Hy =(V,E)$ denotes a hypergraph
where  $V(\Hy)=V$ is the set of vertices and $E(\Hy)=E$ is the set of hyperedges.
\begin{itemize}
 
\item  $\Hy$ is said to be  \emph{linear} if $|e\cap e'|\leq1$ for any pair of distinct hyperedges $e,e'\in E$.

\item  $\Hy$ is said to be \emph{$k$-uniform} if $|e|=k$ for any $e\in E$.

\item We denote by $\mathrm{ar}(\Hy)$ the \emph{antirank} of $\Hy$, namely the minimum cardinality of hyperedges of $\Hy$. We have  $\mathrm{ar}(\Hy)\ge2$ if and only if $H$ has no loop. If $\Hy$ has no hyperedge then $\mathrm{ar}(\Hy)$ is fixed to be equal to $\infty$. The \emph{rank} of $\Hy$ is defined by 
$\mathrm{r}(\Hy):=\max\{|e|,\ e\in E\}$.

\item For $x\in V$ let $\Hy (x)=\{e\in E : e\ni x\}$ denote the \emph{star centered at $x$}.

\item For $e\in E$, $\Hy \smallsetminus e$ denotes the partial hypergraph $(V,E\smallsetminus\{e\})$.

\item The \emph{degree of a vertex} $x$ in $\Hy$ is defined by $\mathrm{deg}_\Hy (x)=|\{e\in E : x\in e\}|$. The \emph{minimum degree} is denoted by $\delta(\Hy)$ and the \emph{maximum degree} by $\Delta(\Hy)$. We have
the following identity
\begin{equation}\label{eqiden}
\sum_{x\in V}\mathrm{deg}_\Hy (x)=\sum_{e\in E}|e|,
\end{equation}
which implies
\begin{equation}\label{eqiden2}
|E|\le \frac{|V|\Delta(\Hy)}{\mathrm{ar}(\Hy)}.
\end{equation}
When $\Hy$ is linear and
if $\delta(\Hy)\ge2$ and $\mathrm{ar}(\Hy)\ge2$, we have  the uniform upper bounds
\begin{equation}\label{eqdex}
\forall x\in V,\quad 
\de(x)\le \frac{|V|-1}{\mathrm{ar}(\Hy)-1},
\end{equation}
$$
\forall e\in E, \quad |e|\le \frac{|E|-1}{\delta(\Hy)-1}.
$$

\item
$\Hy$ is said \emph{$d$-regular} if $\mathrm{deg}_{\Hy}(x)=d$ for any $x\in V$.

\item
The \emph{degree of a hyperedge} $e$ in  $\Hy$ is given by
$$
\mathrm{d}_\Hy (e)=|\{a\in E\smallsetminus\{e\}:  a\cap e\ne\varnothing\}|.
$$
Note that this definition is different from that sometimes given in the literature
(see for instance \cite{BrettoHypergraph2013}). We have
the following bounds:
\begin{equation}\label{eqidenhyp}
\mathrm{d}_\Hy (e)\le |E|-1\quad \text{and}\quad
\sum_{e\in E}\mathrm{d}_\Hy (e)
\le |E|(|E|-1).
\end{equation}
Moreover when $\Hy$ is linear we can develop $\mathrm{d}_\Hy (e)$ as
$$
\mathrm{d}_\Hy (e)=\sum_{x\in e}\sum_{\substack{a\in E\smallsetminus\{e\}\\a\cap e=\{x\}}}1=
\sum_{x\in e}\sum_{\substack{a\in E\smallsetminus\{e\}\\a\ni x}}1
$$
thus
\begin{equation}\label{eqidenhyp2}
\mathrm{d}_\Hy (e)=\sum_{x\in e}(\mathrm{deg}_{\Hy}(x)-1)
\quad\text{and}\quad
\sum_{e\in E}\mathrm{d}_\Hy (e)=\sum_{x\in V}(\mathrm{deg}_{\Hy}(x)-1)\mathrm{deg}_{\Hy}(x).
\end{equation}
\item The \emph{2-section} of  $\Hy$ is the simple graph, denoted by $[\Hy]_{2}$,   whose vertices are 
 those of $\Hy$ and for which two distinct vertices form an edge if and
only if they both belong to a common hyperedge $e$ of $\Hy$.
For any $x\in V$ we have
\begin{equation}\label{eqDH2}
\mathrm{deg}_{[\Hy]_{2}}(x)=\sum_{\substack{e\in E\\e\ni x}}(|e|-1),
\end{equation}
giving 
the following lower bound on the maximum degree in $[\Hy]_{2}$
\begin{equation}\label{eqdelta}
\Delta([\Hy]_{2})\ge (\mathrm{ar}(\Hy)-1)\Delta(\Hy)
\end{equation}
and, when $\Hy$ is linear, the following  identity
\begin{equation}\label{eqdegH2}
\sum_{x\in V}\mathrm{deg}_{[\Hy]_{2}}(x)
=\sum_{e\in E}(|e|^2-|e|).
\end{equation}

\item 
{The \emph{line-graph} of $\Hy$ denoted by $\mathrm{L}(\Hy)$ 
is the simple graph whose vertices are hyperedges of $\Hy$ and
such that there is an edge between $e,e'\in E$, $e\ne e'$, if $e\cap e'\ne \varnothing$. We have $\mathrm{deg}_{\mathrm{L}(\Hy)}(e)=\mathrm{d}_\Hy (e)$
for any $e\in E$.}

\item  Let $A$ be the incidence matrix of the hypergraph $\Hy$ supposed without any isolated vertex and $A^{t}$ be its
  transpose. The hypergraph whose incidence matrix is $A^{t}$ is called the \emph{dual}
   hypergraph of $\Hy$ and denoted by $\Hy ^{*} = (V^{*}, E^{*})$. 
   Clearly $\Hy ^*$ is without isolated vertex and  $\left(\Hy ^*\right)^* = \Hy$. We may notice
that  the second identity in \eqref{eqidenhyp2} is relied to \eqref{eqdegH2} by duality. Note also that $\delta(\Hy) = \mathrm{ar}(\Hy ^*)$.

\item A \emph{$k$-coloring of hyperedges or hyperedge $k$-coloring} of a
 hypergraph $\Hy$ is the assignment of one color from the set $\{1,2,3,\dots, k\}$ to every hyperedge of $\Hy$ in such a way that no two intersecting hyperedges have the same color. 
 

\item The \emph{chromatic index} of a hypergraph $\Hy$, denoted by
 $\mathrm{q}(\Hy)$, is  the least integer $k$ such that there exists a hyperedge $k$-coloring of $\Hy$. It is not  difficult to see that (cf. \cite{BrettoHypergraph2013})
\[
\delta(\Hy) \leq \Delta(\Hy)\leq \Delta_{0}(\Hy)\leq \mathrm{q}(\Hy),
\]
where $\Delta_{0}(\Hy)$ denotes the cardinality of the largest subset $S$ of $E$ such 
that $e,e'\in S\implies e\cap e'\ne\varnothing$.

\item A hyperedge $e\in E$ is said \emph{critical} in $\Hy$ if $\mathrm{q}(\Hy \smallsetminus e)=\mathrm{q}(\Hy)-1$. $\Hy$ is itself said \emph{critical} if all its hyperedges are critical.

\item Let $\Gamma$ be a simple graph.  A \emph{$k$-coloring} of the vertices of $\Gamma$ is an assignment of one color from the set $\{1,2,3,\dots,k\}$ to every vertex of the graph such that no two adjacent vertices have the same color.
The smallest $k$ such that $\Gamma$ has a $k$-coloring, denoted by 
$\chi(\Gamma)$, is called the
 \emph{chromatic number} of $\Gamma$.\\
{It is known (see for instance \cite[Corollary 7.1.7, p. 235]{BFH2022}) 
that $\chi(\Gamma)\leq  \Delta(\Gamma)+1$, 
thus we have the upper bound
\begin{equation}\label{lineg}
\mathrm{q}(\Hy)=\chi(\mathrm{L}(\Hy))\le \Delta(\mathrm{L}(\Hy))+1
\le \mathrm{r}(\Hy)(\Delta(\Hy)-1)+1.
\end{equation}
Hence if $\Hy$ is $k$-uniform, that is $\mathrm{r}(\Hy)=\mathrm{ar}(\Hy)=k$, then by \eqref{eqdelta} the bound \eqref{eqHle} holds whenever $\mathrm{ar}(\Hy)\ge\Delta(\Hy)$. By Theorem \ref{theo15} and the remark that follows, the uniformity of $\Hy$ is actually not required.}

\item The \emph{strong chromatic number} of a hypergraph  $\Hy$  is
the chromatic number of its $2$-section. We denote it by  $\chi(\Hy)$ and we have  $\chi(\Hy)=\chi([\Hy]_{2})$.
Notice that the \emph{(weak) chromatic number} of a hypergraph $\Hy$ (not used in this article) is the least integer $k$
for which  there exists a vertex $k$-coloring of $\Hy$
such that  any hyperedge of $\Hy$ is not monochromatic.

\item In a graph or a hypergraph, we say that a color $i$ is \emph{incident to} a vertex $x$ if there is a (hyper-)edge with color $i$ which contains $x$.

\item Two linear hypergraphs $\Hy =(V,E)$ and $\Hy'=(V',E')$
without loop  are said to be \emph{isomorphic} if there exists a bijection $f$ from $V$ onto $V'$
such that 
$$
\{x_1,\dots,x_t\}\in E\iff \{f(x_1),\dots,f(x_t)\}\in E'.
$$
In this case it is clear that $\mathrm{q}(\Hy)=\mathrm{q}(\Hy')$. 

\item 
{$\Hy$ is said to be a \emph{finite affine plane} if there exists $k\ge2$ such that 
\begin{itemize}
\item $|V|=k^2$ and $\forall x\in V,\ \mathrm{deg}_{\Hy}(x)=k+1$;
\item $|E|=k^2+k$ and $\forall e\in E,\ |e|=k$;
\item $\forall x,y\in V$, $x\ne y$, there exists a unique hyperedge $e$ such that 
$\{x,y\}\subset e$;
\item $\forall a\in E$ and $x\in V\smallsetminus a$, there exists a unique hyperedge $e$ such that $x\in e$ and $e\cap a=\varnothing$.
\end{itemize}
Hence $\Hy$ is a linear $k$-uniform and $(k+1)$-regular hypergraph. Moreover
$E$ can be partitioned into $k+1$ sets of $k$ disjoint hyperedges. This implies  that $\mathrm{q}(\Hy)=k+1$.\\
When $k$ is a prime power, let denote by $\mathbb{F}_k$ the Galois field with cardinality $k$. By identifying lines with hyperedges the \emph{field plane} $\mathbb{F}_k^2$ is a finite affine plane and denoted by $\mathcal{A}_k$.\\
 For a same prime power $k$, it may exist several non-isomorphic finite affine planes (for instance when $k=9$) with $k^2$ vertices. When $k$ is not a prime power it is conjectured that there is no finite affine plane with $k^2$ vertices (known as the Prime Power Conjecture).}
\end{itemize}

\section{\bf Proofs and remarks}
\label{S3}

\subsection{When all hyperedges are large}\label{par31}

The following result shows that the number hyperedges in a linear hypergraph hugely depends on the antirank in a anti-proportional way. Before we notice that \eqref{eqiden2} and \eqref{eqdelta} give together 
$
|E|\le \frac{n}{k^2-k}\Delta([\Hy]_2)
$
(with the notation of the next lemma) without any restriction on $2\le k\le n$.

\begin{lem}\label{lem31}
Let $\Hy =(V,E)$ be a linear  hypergraph with $n$ vertices and $k=\mathrm{ar}(\Hy)$ be its antirank. If $k^2>n$ then
\begin{equation}\label{eq|E|}
|E|\le \frac{n(k-1)}{k^2-n},
\end{equation}
and
\begin{equation}\label{eqDELTAH}
\Delta(\Hy)\le k.
\end{equation}
\end{lem}

\begin{proof}
We may assume that 
\begin{equation}\label{n2k}
|E|\ge\frac{n}{2k}
\end{equation}
since otherwise \eqref{eq|E|} is plainly satisfied whenever $k^2>n$.
From \eqref{eqidenhyp2}, Cauchy inequality and  \eqref{eqiden}, we infer
$$
\sum_{e\in E}(\mathrm{d}_\Hy (e)+|e|)=\sum_{x\in V}\mathrm{deg}_{\Hy}(x)^2\ge
\frac 1n \left(\sum_{x\in V}\mathrm{deg}_{\Hy}(x)\right)^2=
\frac 1n \left(\sum_{e\in E}|e|\right)^2
$$
Thus
$$
\sum_{e\in E}\mathrm{d}_\Hy (e)\ge n(\lambda^2-\lambda)
$$
where $\lambda=\lambda(\Hy)=\frac1n\sum_{e\in E}|e|$.
Since $\lambda\ge \frac{k|E|}{n}$, we have $\lambda\ge1/2$ by \eqref{n2k}, thus
$\sum_{e\in E}\mathrm{d}_\Hy (e)\ge k|E|(\frac{k|E|}n-1)$.
We obtain  by \eqref{eqidenhyp}
$$
|E|(|E|-1)\ge  k|E|\left(\frac{k|E|}n-1\right)
$$
giving
$$
\left(
\frac{k^2}n-1
\right)|E|\le k-1.
$$
This yields \eqref{eq|E|}.
{Bound \eqref{eqDELTAH} follows from 
\eqref{eqdex}.}
\end{proof}

As a direct consequence of \eqref{eq|E|} we get by  \eqref{eqdelta}
$$
\mathrm{q}(\Hy)\le |E|\le \frac{n}{k^2-n}\times \frac{\Delta([\Hy]_2)}{\Delta(\Hy)}
$$
whenever $k^2>n$.  If in addition $n=o(k^2)$ then $\mathrm{q}(\Hy)=o\left( \frac{\Delta([\Hy]_2)}{\Delta(\Hy)}\right)$ as $n\to\infty$
(see also~§\ref{sub34}). 
{Note also that \eqref{eqDELTAH} implies 
\eqref{cond*}. So  when $k^2>n$, \eqref{eqHle} will follow from Theorem \ref{theo15}. In order to deduce 
Theorem \ref{theo16} from it,  we essentially need to consider the equality case $n=k^2$ (see \S \ref{sub34}).}

\subsection{Proof of Theorem \ref{theo15}}

We first show a simple but useful lemma related to critical hyperedges.
 
\begin{lem}\label{lem32}
Let $\Hy =(V,E)$ be a hypergraph without loop and $e\in E$ be a critical hyperedge, that is 
$\mathrm{q}(\Hy \smallsetminus e)=\mathrm{q}(\Hy)-1$. Then
$\mathrm{q}(\Hy)-1\le \mathrm{d}_\Hy (e)$.  
\end{lem}

\begin{proof}
Assume by contradiction that $\mathrm{q}(\Hy)-2\ge \mathrm{d}_\Hy (e)$. Let $k=\mathrm{q}(\Hy)-1$ and $C=\{c_1,\dots,c_{k}\}$
be the set of colors which are necessary to color the hyperedges of $\Hy \smallsetminus e$. In order to color
$\Hy$ is remains to fix a permitted color for $e$. Since $e$ intersects $\mathrm{d}_\Hy (e)\le \mathrm{q}(\Hy)-2$
hyperedges, there exists $c_i\in C$ which is not used for coloring those hyperedges.
One then assigns the color $c_i$ to $e$  and we get $\mathrm{q}(\Hy)\le k$, a contradiction.  
\end{proof}

We continue with two easy properties in regard to Condition \eqref{cond*}.
\begin{lem}\label{lem33}
Let $\Hy =(V,E)$ be a hypergraph satisfying \eqref{cond*}. Then
\begin{equation}\label{cond**}
\forall e\in E \text{ such that }|e|=\mathrm{ar}(\Hy),\ \exists x\in e\text{ such that }\mathrm{deg}_{\Hy}(x)\le \mathrm{ar}(\Hy),
\end{equation}
and for any $e'\in E$ the partial hypergraph $\Hy'=\Hy \smallsetminus e'=(V,E\smallsetminus\{e'\})$ also satisfies~\eqref{cond*}.
\end{lem}

\begin{proof}
The implication \eqref{cond*}$\implies$ \eqref{cond**} is plain. 
\\
Let $e\in E\smallsetminus\{e'\}$. Since $\Hy$ satisfies \eqref{cond*}, there exists $x\in V$ such that $ \mathrm{deg}_{\Hy}(x)\le |e|$. We easily conclude since  $\mathrm{deg}_{\Hy'}(x)\le \mathrm{deg}_{\Hy}(x)$. 
\end{proof}

\begin{rem}\label{rm34}
We stress that fact that Condition \eqref{cond*} is hereditary, namely 
$$
 \Hy  \text{ satisfies \eqref{cond*}}\implies \Hy \smallsetminus e'  \text{ satisfies \eqref{cond*}},
$$
but it is no longer the case with Condition \eqref{cond**}. This hereditary property will play a central role along the induction argument  in the proof of Theorem \ref{theo15}.
\end{rem}

\begin{proof}[\textcolor{blue}{Proof of Theorem \ref{theo15}}]
We  argue by induction on the number of hyperedges. \\
If $|E|=1$, clearly $\mathrm{q}(\Hy)=1\le \Delta([\Hy]_2)+1$ and we are done.\\
Let $m\ge1$ and $\Hy =(V,E)$ with $|E|=m+1$ hyperedges. \\
Let $e_0$ be such that $|e_0|=\mathrm{ar}(\Hy)$ and $x_0\in e_0$ such that $\mathrm{deg}_{\Hy}(x_0)\le |e_0|=\mathrm{ar}(\Hy)$. Let $\Hy _0=\Hy \smallsetminus e_0$. Clearly $\Hy _0$ has $m$ hyperedges and satisfies \eqref{cond*} since
$\mathrm{deg}_{\mathcal{\Hy}_0}(x)\le \mathrm{deg}_{\Hy}(x)$. 
Therefore we may apply our induction hypothesis to $\Hy _0$. We distinguish two cases.

\begin{itemize}
\item If $\mathrm{q}(\Hy _0)=\mathrm{q}(\Hy)-1$ then by Lemma \ref{lem32} and \eqref{eqdelta} we  get
\begin{align*}
\mathrm{q}(\Hy)-1\le \mathrm{d}_\Hy (e_0)&=\sum_{\substack{x\in e_0\\x\ne x_0}}(\mathrm{deg}_{\Hy}(x)-1)
+\mathrm{deg}_{\Hy}(x_0)-1\\
&\le (|e_0|-1)\left(\max_{x\in e_0}\mathrm{deg}_{\Hy}(x)-1\right)+|e_0|-1\\
&= (\mathrm{ar}(\Hy)-1)\max_{x\in e_0}\mathrm{deg}_{\Hy}(x)\\
&\le (\mathrm{ar}(\Hy)-1)\Delta(\Hy)\\
&\le \Delta([\Hy]_2).
\end{align*}

\item Otherwise $\mathrm{q}(\Hy _0)=\mathrm{q}(\Hy)$ and by induction hypothesis we get
$$
\mathrm{q}(\Hy)=\mathrm{q}(\Hy _0)\le  \Delta([\Hy _0]_2)+1 \le \Delta([\Hy]_2)+1.
$$
\end{itemize}
This finishes the proof of Theorem \ref{theo15}.
\end{proof}

\subsection{Proof of Theorem \ref{theo16}}\label{sub34}

If $\Hy =(V,E)$ satisfies \eqref{cond*}  then by Theorem \ref{theo15} we infer the desired result. Assume the contrary, that is there exists $e_0\in E$ 
such that 
\begin{equation}\label{degHx}
\forall x\in e_0 \quad \mathrm{deg}_{\Hy}(x)\ge |e_0|+1\ge \mathrm{ar}(\Hy)+1
=k+1
\end{equation}
where 
\begin{equation}\label{k=}
k=\mathrm{ar}(\Hy).
\end{equation}
Our assumption on the antirank of $\Hy$  can be rewritten as
\begin{equation}\label{kV}
k\ge|V|^{1/2}.
\end{equation}
We shall use the notation $S=\{a\in E\,:\, a\cap e_0=\varnothing\}$ the set of hyperedges which are \emph{parallel} to $e_0$ and let $s=|S|$. Then
\begin{equation}\label{eq|E|16}
|E|=\mathrm{d}_\Hy (e_0)+1+s.
\end{equation}

\medskip
\noindent
\textit{Step 1}:
let $x\in e_0$. By \eqref{degHx}, \eqref{k=} and \eqref{kV} we have following inequalities:
\begin{multline*}
k^2\ge|V|\ge \mathrm{deg}_{[\Hy]_2}(x)+1=
\sum_{\substack{e\in E\\e\ni x}}(|e|-1)+1=
\sum_{\substack{e\in E\smallsetminus\{e_0\}\\e\ni x}}(|e|-1)+|e_0|\\
\ge 
(\mathrm{deg}_{\Hy}(x)-1)(k-1)+k\ge k(k-1)+k=k^2.
\end{multline*}
This implies that the above inequalities are all equalities. This gives
\begin{itemize}
\item $|V|=k^2$,
\item $\mathrm{deg}_{[\Hy]_2}(x)=k^2-1$,
\item $\mathrm{deg}_{\Hy}(x)=k+1$,
\end{itemize}
We infer
\begin{equation}\label{de0}
\mathrm{d}_\Hy (e_{0})=\sum_{x\in e_{0}}(\mathrm{deg}_{\Hy}(x)-1)= k^{2},
\end{equation}
and 
\begin{equation}\label{|E|b}
|E|\ge \mathrm{d}_\Hy (e_0)+1=k^2+1.
\end{equation}

\medskip\noindent
\textit{Step 2}: we prove that $\Hy$ is $k$-uniform and 
that  $\forall y\in V\smallsetminus e_0,\ \mathrm{deg}_{\Hy}(y)\ge k$.

\begin{itemize}
\item  $|e_0|=k$ by \eqref{degHx}, since $\mathrm{deg}_{\Hy}(x)=k+1$ for $x\in e_0$.

\item Let $a\in E\smallsetminus (S\cup\{e_0\})$. Then $a\cap e_0=\{x\}$
for some vertex $x\in e_0$.


As in Step 1, we may again write
$$
k^2-1=\mathrm{deg}_{[H]_{2}}(x)=\sum_{\substack{e\in E\\e\ni x}}
(|e|-1)=\sum_{\substack{e\in E\smallsetminus\{a\}\\e\ni x}}(|e|-1)+|a|-1
$$
thus
$$
k^2-1\ge (\mathrm{deg}_{\Hy}(x)-1)(k-1)+|a|-1=k(k-1)+|a|-1.
$$
This implies $|a|\le k$ yielding $|a|=k$ by \eqref{k=}.

\item Let $e_0=\{x_0,x_1,\dots,x_{k-1}\}$. Then for any $0\le i\le k-1$, we have $\mathrm{deg}_{\Hy}(x_i)=k+1$. We let $\Hy (x_i)=\{e_0,e_{i,1},\dots,e_{i,k}\}$ be the star centered at $x_i$, $0\le i\le k-1$. 
From above $|e_{i,j}|=k$, $1\le j\le k$, thus for any $0\le i\le k-1$, we have
$$
\left|e_0\cup\bigcup_{j=1}^k e_{i,j}\right|=
\sum_{j=1}^k(|e_{i,j}|-1)+|e_0|=k(k-1)+k=k^2=|V|,
$$
giving
\begin{equation}\label{V=}
V=e_0\cup\bigcup_{j=1}^ke_{i,j},\quad 0\le i\le k-1.
\end{equation}
This implies that any vertex $y$ in $V\smallsetminus\{e_0\}$ is adjacent to $x_i$ for any $0\le i\le k-1$. 
{Hence for such $y$ there exist distinct hyperedges $a_i$, $0\le i\le k-1$, such that both $x_i$ and $y$ are in $a_i$. This gives $\mathrm{deg}_{\Hy}(y)\ge k$.} 

\item Let $a\in S$. We see by \eqref{V=} that any vertex of $a$ is in some $e_{0,j}\smallsetminus\{x_0\}$, $1\le j\le k$. Since $\Hy$ is linear, we have $|a\cap e_{0,j}|\le1$, $1\le j\le k$. We obtain $|a|\le k$ and finally $|a|=k$  by \eqref{k=}.
\end{itemize}
We infer that $\Hy$ is $k$-uniform.

\medskip\noindent
\textit{Step 3}: let $a,b\in S$, $a\ne b$. 
 If $y\in a\cap b$, then $y$ is adjacent to 
each $x_i$, $0\le i\le k-1$ through $k$ many hyperedges. Therefore considering in addition $a$ and $b$, it follows that $y$ is contained in $k+2$ many hyperedges, a contradiction. Hence $a\cap b=\varnothing$. In other words 
two hyperedges parallel to $e_0$ are indeed \emph{parallel}.


Since $\Hy$ is $k$-uniform we have
$k^2=|V|\ge |e_0|+\sum_{a\in S}|a|=k+ks$ hence $0\le s\le k-1$. 
By \eqref{eq|E|16} and \eqref{de0} we infer 
$$
|E|=\mathrm{d}_\Hy (e_0)+1+s= k^2+s+1.
$$

\medskip\noindent
\textit{Step 4}: 
 the partial hypergraph $\tilde{\Hy}=(V,E\smallsetminus S)$
belongs to the class  $\mathrsfs{H}_k$ of all hypergraphs 
$\tilde{\Hy}=(\tilde{V},\tilde{E})$ satisfying
\begin{itemize}
\item  $\tilde{\Hy}$ is linear and $k$-uniform,

\item $|\tilde{V}|=k^2$,

\item $|\tilde{E}|=k^2+1$,

\item and there exists $e_0\in \tilde{E}$ such that\\
-- $\forall x\in e_0$, $\mathrm{deg}_{\tilde{\Hy}}(x)=k+1$,\\
--  $\forall x\in V\smallsetminus\{e_0\}$,  $\mathrm{deg}_{\tilde{\Hy}}(x)\ge k$.
\end{itemize}
If we have a hyperedge coloring of $\tilde{\Hy}$, then assigning
to all hyperedges in $S$ the same color already assigned to $e_0$ provides  a hyperedge coloring of 
$\Hy$. Thus $\mathrm{q}(\Hy)=\mathrm{q}(\tilde{\Hy})$. 
It remains to show that $\mathrm{q}(\tilde{\Hy})\le \Delta([\Hy]_2)+1$. 
This will follow from Theorem \ref{theo36} iii) below. Note that
it is possible to get this bound by a shorter argument but instead we make the choice  to study more widely the class  $\mathrsfs{H}_k$.

Before stating and proving Theorem \ref{theo36} we outline  \emph{geometric} properties for hypergraphs of $\mathrsfs{H}_k$.

\begin{prop}\label{pp35}
Let $k\ge2$, $\Hy =(V,E)\in\mathrsfs{H}_k$,
$e_0\in E$ satisfy \eqref{degHx} and $e\in E\smallsetminus\{e_0\}$. 
Then
\begin{enumerate}

\item We have $\mathrm{d}_\Hy (e_0)=k^2$ and for any $a\in E$, $a\cap e_0\ne\varnothing$. 
In other words there is no hyperedge parallel to $e_0$.

\item There exists $x_0\in e_0$ such that $e\cap e_0=\{x_0\}$.

\item For any $x\in e_0\smallsetminus \{x_0\}$ 
there exists a unique hyperedge $a\in \Hy (x)\smallsetminus\{e_0\}$ such that
$a\cap e=\varnothing$.


\item Assume $k\ge3$. It may exist  $x\in V\smallsetminus (e_0\cup e)$ and distinct hyperedges $a,a'$ 
such that $a\cap e=a'\cap e=\varnothing$ and $a\cap a'=\{x\}$.

\item It is not true in general that $a\cap e=a'\cap e=\varnothing\implies a\cap a'=\varnothing$.
\end{enumerate}
\end{prop}

\begin{proof}
i) By \eqref{de0} and since $|E|=k^2+1$, all hyperedges intersect $e_0$.

\smallskip\noindent
ii) follows from i).

\smallskip\noindent
iii) Existence: otherwise for any hyperedge $a\ne e_0$ and containing $x$, we have
$a\cap e=\{y_a\}$ for some vertex $y_a\in e\smallsetminus \{x_0\}$. Since $\Hy$ is linear
the mapping $a\in \Hy (x)\smallsetminus\{e_0\}\mapsto y_a\in e\smallsetminus\{x_0\}$ is injective. But $|\Hy (x)|=\mathrm{deg}_{\Hy}(x)=k+1$, hence $|e|\ge k+1$, a contradiction.\\
Uniqueness: let $a$ satisfy the desired conclusion. Then $a\cap e_0=\{x\}$
and $a\cap e=\varnothing$.
By~ii),  each of the  $k-1$ many vertices of $e\smallsetminus\{x_0\}$ lies in exactly one of the $k-1$ many hyperedges of $\Hy (x)\smallsetminus\{e_0,a\}$. This gives a mapping $y\in e\smallsetminus\{x_0\}\mapsto a_y\in 
\Hy (x)\smallsetminus\{e_0,a\}$ which is injective by linearity of 
$\Hy$. Since $|e\smallsetminus\{x_0\}|=k-1=|\Hy (x)\smallsetminus\{e_0,a\}|$ it is a bijection. It follows that for any 
$a'\in \Hy (x)\smallsetminus\{e_0,a\}$, there exists $y\in 
e\smallsetminus\{x_0\}$ such that $a'\cap e=\{y\}$.

\smallskip\noindent
iv) and v) will follow from the construction of $\Hy'$ in iv) of Proposition \ref{pp42}.
\end{proof}

\begin{thm}\label{theo36}
For any $k\ge2$ and any $\Hy \in\mathrsfs{H}_k$, we have 
\begin{enumerate}
\item $\Delta([\Hy]_2)=k^2-1$ and $\delta([\Hy]_2)=k^2-k$; 
\item  $k+1\le \mathrm{q}(\Hy)\le 1+k\lceil \frac{k}2\rceil$;
\item  $\mathrm{q}(\Hy)\le \Delta([\Hy]_2)$.
\end{enumerate}
\end{thm}

\begin{proof} 
i)  Let  $x_i\in e_0$, then $x_i$ is adjacent to any other vertex of $V$. Thus
$\mathrm{deg}_{[\Hy]_2}(x_i)=k^2-1$.  
If $x\in V\smallsetminus e_0$
then $x$ belongs to $k$ different hyperedges thus admits $k^2-k$ adjacent vertices. We infer  
$\Delta([\Hy]_2)=k^2-1$ and $\delta([\Hy]_2)=k^2-k$ .\\[0.3em]
ii)  We fix $c_0$ the color of $e_0$, and for each even $0\le i\le k-1$ and any $0\le j\le k$ we color $e_{i,j}$ with $c_{i,j}$. We need 
exactly 
$$
1+\sum_{\substack{0\le i\le k-1\\i\text{ even}}}k=1+k\left\lceil\frac{k}{2}\right\rceil 
$$
many different colors. Our aim is to show that they are sufficiently numerous to provide a hyperedge coloring of $\Hy$. For this it remains to color
all hyperedges $e_{i,h}$ where $1\le i\le k-1$ is odd and $1\le h\le k$.  

\medskip
We fix an odd integer $1\le i\le k-1$. 
\begin{itemize}
\item For each $1\le j\le k$ there
exists a unique $1\le h=h(j)\le k$ such that $e_{i-1,j}\cap e_{i,h}=\varnothing$: it suffices to apply iii) of Proposition \ref{pp35} with 
$e=e_{i-1,j}$ and $x=x_i$.

We thus may define the map
$$
f:
\begin{array}{ccc}
\{e\in E\smallsetminus \{e_0\}\,:\, x_{i-1}\in e\} &\longrightarrow & \{e\in E\smallsetminus \{e_0\}\,:\, x_{i}\in e\}\\
e_{i-1,j}&\longmapsto& e_{i,h(j)}
\end{array}
$$

\item $f$ is injective. Otherwise it would exist $1\le  j\ne j'\le k$ such that
$h(j)=h(j')$. Hence
$$
e_{i,h(j)}\smallsetminus\{x_{i}\}\subset
\bigsqcup_{\substack{1\le t\le k\\t\ne j,j'}}e_{i-1,t}
\smallsetminus\{x_{i-1}\}.
$$
Since $|e_{i,h(j)}\smallsetminus\{x_{i}\}|=k-1$ and 
$|(e_{i,h(j)}\smallsetminus\{x_{i}\})\cap e_{i-1,t}
\smallsetminus\{x_{i-1}\}|\le 1$ for any $1\le t\ne j,j'\le k$ we deduce from the \emph{pigeon hole principle} that $|e_{i,h(j)}\cap e_{i-1,t}|\ge2$ for some $t$, a contradiction to the
linearity of $\Hy$.

\item Thus $f$ is bijective, that is for any $1\le h\le k$ there exists $1\le j\le k$
such that $e_{i,h}\cap e_{i-1,j}=\varnothing$. This shows that hyperedge 
$e_{i,h}$ can be colored with $c_{i-1,j}$.
\end{itemize}
Statement iii) follows from i) and ii).
\end{proof}

\begin{rems}\label{rm37}
\begin{enumerate}
\item  There exist several sharp results for the chromatic index of linear
$k$-uniform and $d$-regular hypergraphs $\Hy$ yielding in particular the asymptotic upper bound
$\mathrm{q}(\Hy)\le (1+o(1))d$ when $d\to\infty$, where the implied function in $o(1)$ may depend on $k$ (see \cite[Section 2]{Kang2022} for a rich survey on this topic). However they do not apply when $d=k$. 

\item The size of the hyperedge coloring of $\Hy \in\mathrsfs{H}_k$
given in Theorem \ref{theo36} is certainly not optimal. For instance when $k=3$
we shall see in Proposition \ref{pp42} that $\mathrm{q}(\Hy)\le5$ while 
 the algorithm provided by Theorem \ref{theo36} yields $\mathrm{q}(\Hy)\le7$. It could be interesting to improve the upper bound in ii) and to give the order of magnitude of
$\max\{\mathrm{q}(\Hy),\ {\Hy \in\mathrsfs{H}_k}\}$
in terms of $k$. See Proposition \ref{pp42} for additional properties.

\end{enumerate}
\end{rems}

\section{\bf Studying the class $\mathrsfs{H}_k$}
\label{S4}

\subsection{General facts}

For any integer $k\ge2$, we denote by $\mathrsfs{H}_k$ the set  hypergraphs $\Hy =(V,E)$
such that 
\begin{itemize}
\item[(C0)] $\Hy$ is  linear and $k$-uniform,
\item[(C1)]  $|V|=k^2$,
\item[(C2)] $|E|=k^2+1$,
\end{itemize}
and there exists  a  hyperedge $e_0$ such that
\begin{itemize}
\item[(C3)] $\forall x\in e_0$, $\mathrm{deg}_{{\Hy}}(x)=k+1$,
\item[(C4)] $\forall x\in V\smallsetminus e_0$, $\mathrm{deg}_{{\Hy}}(x)\ge k$.
\end{itemize}

\medskip
We start with easy consequences of the definition.

\begin{prop} \label{pp41}
Let $\Hy =(V,E)$ satisfy Conditions (C0)-(C4).
Then
\begin{enumerate}
\item $\forall x\in V\smallsetminus e_0$, $\mathrm{deg}_{\Hy}(x)=k$,



\item $\forall e\in E\smallsetminus\{e_0\}$, $\mathrm{d}_\Hy (e)=k^2-k+1$.
\end{enumerate}
\end{prop}

\begin{proof}
 i) By \eqref{eqiden} and (C2), we have $\sum_{x\in V}\mathrm{deg}_{\Hy}(x)=(k^2+1)k$, hence by (C3) and (C4)
\begin{align*}
k^3+k=(k+1)k+(k^2-k)k&\le (k+1)k+(k^2-k)\min_{x\in V\smallsetminus e_0}\mathrm{deg}_{\Hy}(x)
\\&\le 
(k+1)k+\sum_{x\in V\smallsetminus e_0}\mathrm{deg}_{\Hy}(x)=(k^2+1)k,
\end{align*}
since there are $(k^2-k)$ vertices not belonging to  $e_0$.
It follows that $\forall x\in V\smallsetminus e_0$, $\mathrm{deg}_{\Hy}(x)=k$.

\medskip
\noindent
ii)
For any $e\in E\smallsetminus\{e_0\}$,  we get from from the first equation in \eqref{eqidenhyp2}, (C3) and i)
$$
\mathrm{d}_\Hy (e)=\sum_{x\in e\cap e_0}k+\sum_{x\in e\smallsetminus e_0}(k-1)=k+(k-1)^2=k^2-k+1
$$
since $|e\cap e_0|=1$ and $|e\smallsetminus e_0|=k-1$.
\end{proof}
 
\begin{prop}\label{pp42}
For any prime power $k$ there exists a hypergraph denoted by 
$\widehat{\mathcal{A}_k}$  in the class $\mathrsfs{H}_k$ such that
\begin{enumerate}
\item  $\mathrm{q}(\widehat{\mathcal{A}_k})=k+1$,
\item $\widehat{\mathcal{A}_k}$ is not maximal
 (for inclusion) in the class of all linear and $k$-uniform
 hypergraphs with $k^2$ vertices,
\item $\widehat{\mathcal{A}_k}$ contains only one critical hyperedge,
namely the unique hyperedge $e_0$ satisfying (C3) and (C4).
\end{enumerate}
\end{prop}

\begin{proof}
The field plane $\mathcal{A}_k=\mathbb{F}_k^2$ over $\mathbb{F}_k$ provides a $(k+1)$-regular and $k$-uniform linear hypergraph with $k^2$ vertices when considering the $k^2+k$ lines as its hyperedges.
By removing any $k-1$ non intersecting hyperedges (namely parallel lines) from $\mathcal{A}_k$ we obtain
a hypergraph $\widehat{\mathcal{A}_k}$ of $\mathrsfs{H}_k$. \\[0.3em]
i) Lines of $\mathcal{A}_k$ can be partitioned according to their slope. For each slope, any point
appears only once on some line of $\mathcal{A}_k$. This shows that all the lines having a common
slope can be colored with the same color.
Since there are $k+1$ distinct slopes in $\mathcal{A}_k$, including the vertical one, we can conclude.\\[0.3em]
ii) This follows from the construction of $\widehat{\mathcal{A}_k}$.\\[0.3em]
iii) The unique line without any parallel line in $\widehat{\mathcal{A}_k}$  is
obviously critical, while removing any other line does not decrease the chromatic index. 
\end{proof}

\subsection{The cases $k=2,3$}

\begin{thm}\label{theo43}
We have:
\begin{enumerate}
\item $\mathrsfs{H}_2$ contains just the graph $\widehat{\mathcal{A}_2}$ up to isomorphism.
It satisfies $\mathrm{q}(\widehat{\mathcal{A}_2})=3$ and $\mathrm{Aut}(\widehat{\mathcal{A}_2})\simeq\mathfrak{S}_2\times \mathfrak{S}_2$, the Klein four-group.

\item $\mathrsfs{H}_3$ possesses exactly two non isomorphic elements $\widehat{\mathcal{A}_3}$ and $\Hy'_3$ which 
satisfy $\mathrm{q}(\widehat{\mathcal{A}_3})=4$ and $\mathrm{q}(\Hy'_3)=5$.

\item 
$\mathrm{Aut}(\widehat{\mathcal{A}_3})=\mathfrak{S}_3 \rtimes \mathfrak{S}_3$ and $\mathrm{Aut}(\Hy'_3)\simeq\mathfrak{S}_2 \times \mathfrak{S}_3=D_6$, the dihedral group of degree $6$ and order $12$.

\end{enumerate}
\end{thm}

\begin{proof}

i) For $k=2$, we see easily that any hypergraph $\Hy$ in $\mathrsfs{H}_2$ is obtained by removing a single edge from the complete simple graph on four vertices $K_4$, so that $\Hy =\widehat{\mathcal{A}_2}$ and $\mathrm{q}(\Hy)=3$ by Proposition \ref{pp42}. 
Clearly any automorphism of $\Hy$ is allowed to permute both vertices of 
degree 2 in one side, and independently both vertices of degree $3$ on the other. Hence the result. 
\\[0.5em]
ii) Let $(V,E)\in \mathrsfs{H}_3$. We suppose that
$V=\mathbb{F}_3^2$ where we identify the vertex $(x,y)$ with
the integer $3x+y$ whose development in base $3$ is $xy$. 
We may assume that the \emph{special} hyperedge $e_0$ is $\{0,1,2\}$ simply denoted
$012$. The 3 other hyperedges containing $0$ can be set to $036$, $048$ and 
$057$ with the same short-hand notation. Note that at this stage, the vertices $3,6$ can be permuted, likewise for $4,8$ and for $5,7$. 
These 4 hyperedges are lines in the field plane $\mathbb{F}_3^2$. Now the 
3 hyperedges containing $1$ and different from $e_0$ are
$$
13u, 16v, 1xw \quad\text{where $\{u,v,w,x\}=\{4,5,7,8\}$.} 
$$
We must have $x=4$ or $8$ since $\{w,x\}\ne \{5,7\}$ by the facts that hyperedge $057$ is already defined and $\Hy$ is linear. Switching if necessary $4$ and $8$,
we may assume that $x=4$. This implies $w\in\{5,7\}$. Switching if necessary $5$ and $7$ we may set $w=7$. Since $3$ and $6$ are also permutable
we may set $u=8$ and $v=5$. Up to isomorphism, hyperedges
$$
012, 036, 048, 057, 147, 138, 156
$$
are in $E$. We may observe on the following figure that they are lines in $\mathbb{F}_3^2$.

\begin{center}
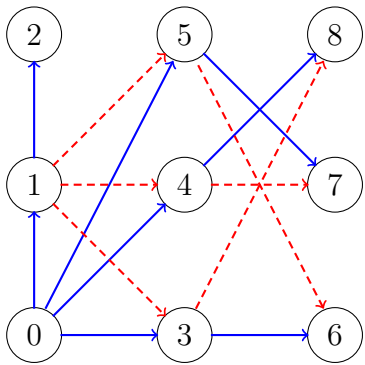
\begin{figure}[h!]
\begin{tikzpicture}


\draw (0,0) node[circle,draw] {$0$};
\draw (0,2) node[circle,draw] {$1$};
\draw (0,4) node[circle,draw] {$2$};
\draw (2,0) node[circle,draw] {$3$};
\draw (2,2) node[circle,draw] {$4$};
\draw (2,4) node[circle,draw] {$5$};
\draw (4,0) node[circle,draw] {$6$};
\draw (4,2) node[circle,draw] {$7$};
\draw (4,4) node[circle,draw] {$8$};

\draw [thick,blue,->] (0,0.35) -- (0,1.65);
\draw [thick,blue,->] (0,2.35) -- (0,3.65);
\draw [thick,blue,->] (0.35,0) -- (1.65,0);
\draw [thick,blue,->] (2.35,0) -- (3.65,0);
\draw [thick,blue,->] (0.25,0.25) -- (1.75,1.75);
\draw [thick,blue,->] (2.25,2.25) -- (3.75,3.75);
\draw [thick,blue,->] (0.15,0.35) -- (1.85,3.65);
\draw [thick,blue,->] (2.25,3.75) -- (3.75,2.25);

\draw [thick,red,densely dashed,->] (0.35,2) -- (1.65,2);
\draw [thick,red,densely dashed,->] (2.35,2) -- (3.65,2);
\draw [thick,red,densely dashed,->] (0.25,1.75) -- (1.75,0.25);
\draw [thick,red,densely dashed,->] (2.15,0.35) -- (3.85,3.65);
\draw [thick,red,densely dashed,->] (0.25,2.25) -- (1.75,3.75);
\draw [thick,red,densely dashed,<-] (3.85,0.35) -- (2.15,3.65);


\end{tikzpicture}
\caption{A partial hypergraph of $\Hy$ with $7$ lines \emph{immersed} in the field plane $\mathcal{A}_3$.}
\end{figure}

\end{center}

\noindent
Let 
$$
2xw, 23u, 26v\quad \text{where }
u\in\{4,5,7\},\ v\in\{4,7,8\},\ 
\{w,x\}=\{\{4,5\},\{5,8\},\{7,8\}\},
$$
be the 3 remaining hyperedges. At least one of them is a line in the plane. Indeed, otherwise we would have
$$
u\in\{4,5\},\ v\in\{7,8\},\ 
\{w,x\}=\{\{4,5\},\{7,8\}\},
$$
a contradiction since $u,v,w,x$ are distinct.\\
Applying if necessary a linear transformation with matrix $\begin{pmatrix}
\alpha & 0\\\beta & 1\end{pmatrix}$ for which the vertices $0,1,2$ are fixed,
we may assume that $258$ is in $E$. We get $\{u,v\}=\{4,7\}$.
\begin{itemize}
\item First case: if the set of hyperedges is
$$
E=\{012, 036, 048, 057, 147, 138, 156, 258, 237, 246\},
$$
all the hyperedges are lines in the plane. We thus obtain the hypergraph 
$\Hy =\widehat{\mathcal{A}_3}$.
\begin{center}

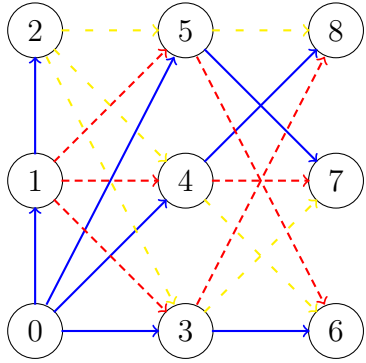
\begin{figure}[H]
\begin{tikzpicture}


\draw (0,0) node[circle,draw] {$0$};
\draw (0,2) node[circle,draw] {$1$};
\draw (0,4) node[circle,draw] {$2$};
\draw (2,0) node[circle,draw] {$3$};
\draw (2,2) node[circle,draw] {$4$};
\draw (2,4) node[circle,draw] {$5$};
\draw (4,0) node[circle,draw] {$6$};
\draw (4,2) node[circle,draw] {$7$};
\draw (4,4) node[circle,draw] {$8$};

\draw [thick,blue,->] (0,0.35) -- (0,1.65);
\draw [thick,blue,->] (0,2.35) -- (0,3.65);
\draw [thick,blue,->] (0.35,0) -- (1.65,0);
\draw [thick,blue,->] (2.35,0) -- (3.65,0);
\draw [thick,blue,->] (0.25,0.25) -- (1.75,1.75);
\draw [thick,blue,->] (2.25,2.25) -- (3.75,3.75);
\draw [thick,blue,->] (0.15,0.35) -- (1.85,3.65);
\draw [thick,blue,->] (2.25,3.75) -- (3.75,2.25);

\draw [thick,red,densely dashed,->] (0.35,2) -- (1.65,2);
\draw [thick,red,densely dashed,->] (2.35,2) -- (3.65,2);
\draw [thick,red,densely dashed,->] (0.25,1.75) -- (1.75,0.25);
\draw [thick,red,densely dashed,->] (2.15,0.35) -- (3.85,3.65);
\draw [thick,red,densely dashed,->] (0.25,2.25) -- (1.75,3.75);
\draw [thick,red,densely dashed,->]  (2.15,3.65) -- (3.85,0.35);

\draw [thick,loosely dashed,->,yellow] (0.35,4) -- (1.65,4);
\draw [thick,loosely dashed,->,yellow] (2.35,4) -- (3.65,4);
\draw [thick,loosely dashed,->,yellow]  (0.15,3.65) -- (1.85,0.35);
\draw [thick,loosely dashed,->,yellow]  (2.25,0.25) -- (3.75,1.75);
\draw [thick,loosely dashed,->,yellow]  (0.25,3.75) -- (1.75,2.25);
\draw [thick,loosely dashed,->,yellow]  (2.25,1.75) -- (3.75,0.25);

\end{tikzpicture}
\caption{The hypergraph $\widehat{\mathcal{A}_3}$.}
\end{figure}
\end{center}

\item
Second case: if the set of hyperedges is
$$
E'=\{012, 036, 048, 057, 147, 138, 156, 258, 234, 276\},
$$
we let $\Hy'_3=(V,E')$.

\begin{center}
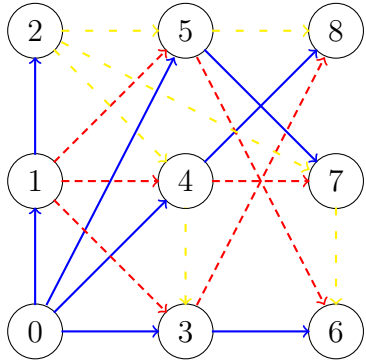
\begin{figure}[H]
\begin{tikzpicture}


\draw (0,0) node[circle,draw] {$0$};
\draw (0,2) node[circle,draw] {$1$};
\draw (0,4) node[circle,draw] {$2$};
\draw (2,0) node[circle,draw] {$3$};
\draw (2,2) node[circle,draw] {$4$};
\draw (2,4) node[circle,draw] {$5$};
\draw (4,0) node[circle,draw] {$6$};
\draw (4,2) node[circle,draw] {$7$};
\draw (4,4) node[circle,draw] {$8$};

\draw [thick,blue,->] (0,0.35) -- (0,1.65);
\draw [thick,blue,->] (0,2.35) -- (0,3.65);
\draw [thick,blue,->] (0.35,0) -- (1.65,0);
\draw [thick,blue,->] (2.35,0) -- (3.65,0);
\draw [thick,blue,->] (0.25,0.25) -- (1.75,1.75);
\draw [thick,blue,->] (2.25,2.25) -- (3.75,3.75);
\draw [thick,blue,->] (0.15,0.35) -- (1.85,3.65);
\draw [thick,blue,->] (2.25,3.75) -- (3.75,2.25);

\draw [thick,red,densely dashed,->] (0.35,2) -- (1.65,2);
\draw [thick,red,densely dashed,->] (2.35,2) -- (3.65,2);
\draw [thick,red,densely dashed,->] (0.25,1.75) -- (1.75,0.25);
\draw [thick,red,densely dashed,->] (2.15,0.35) -- (3.85,3.65);
\draw [thick,red,densely dashed,->] (0.25,2.25) -- (1.75,3.75);
\draw [thick,red,densely dashed,->]  (2.15,3.65) -- (3.85,0.35);

\draw [thick,loosely dashed,->,yellow] (0.35,4) -- (1.65,4);
\draw [thick,loosely dashed,->,yellow] (2.35,4) -- (3.65,4);
\draw [thick,loosely dashed,->,yellow]  (0.25,3.75) -- (1.75,2.25);
\draw [thick,loosely dashed,<-,yellow]  (2,0.35) -- (2,1.65);
\draw [thick,loosely dashed,->,yellow]  (4,1.65) -- (4,0.35);
\draw [thick,loosely dashed,->,yellow]  (0.35,3.85) -- (3.65,2.15);

\end{tikzpicture}
\caption{The hypergraph $\Hy'_3$.}
\end{figure}
\end{center}

\end{itemize}
\noindent
We infer that up to isomorphism the class $\mathrsfs{H}_3$ reduces to $\{\widehat{\mathcal{A}_3},\Hy'_3\}$.

\medskip  
We have $\mathrm{q}(\Hy'_3)\le 5$ since
$$
0 \curvearrowright 012;\quad 1\curvearrowright 036,147,258;\quad 2 \curvearrowright 048,156;\quad 3\curvearrowright 057,234; \quad 4\curvearrowright 138,276;
$$
provides a hyperedge $5$-coloring of $\Hy'_3$.
Now suppose by contradiction that $4$ colors are sufficient. Without loss of generality, we may assume that we have made the following partial assignment of colors:
$$
0 \curvearrowright 012;\quad 1\curvearrowright 036,147;\quad 2 \curvearrowright 048,156;\quad 3\curvearrowright 057,138.
$$
We see that hyperedge $234$ then necessarily requires an additional color, 
contradiction. Hence $\mathrm{q}(\Hy'_3)=5$. 
Since $\mathrm{q}(\widehat{\mathcal{A}_3})=4$ by i), $\Hy'_3$ is not isomorphic to 
$\widehat{\mathcal{A}_3}$. We conclude that $|\mathrsfs{H}_3|=2$.

Note also that $\Hy'_3$ is maximal in the class of $3$-uniform linear 
hypergraphs: adding a new hyperedge yields a non linear hypergraph. In 
contrast 
$\widehat{\mathcal{A}_3}=(V,E)$ can be completed with two additional lines, namely  $345,678$.
This shows by a different way that $\Hy'_3$ and $\widehat{\mathcal{A}_3}$ are 
non isomorphic. 

\medskip
\noindent
iii) We use notation of ii) and only sketch the proof.  Since $e_0$ contains all the 
vertices of degree $k+1$ in $\Hy$, the subset $e_0$ of $V$ is fixed by any 
automorphim 
$\varphi$  of $\Hy$. Denoting by $\varphi_{|e_0}$ the restriction of $\varphi$ to 
$e_0$, we get  in both case
 $\Hy=\widehat{\mathcal{A}_3},\Hy'_3$, the homomorphism
$$
g:
\begin{aligned}
\mathrm{Aut}(\Hy)&\longrightarrow&\mathfrak{S}_3\\
\varphi & \longmapsto&\varphi_{|e_0}
\end{aligned}
$$ 
\begin{itemize}

\item When $\Hy=\widehat{\mathcal{A}_3}$ we have 
$$
\ker g=\langle p,t\rangle\simeq\mathfrak{S}_3
\quad \text{where\quad $p=(4\ 3\ 5)(6\ 7\ 8)$ and $t=(4\ 6)(5\ 7)(3\ 8)$}
$$
with usual notation for permutation running on the vertices $0,1\dots,8$.
We check easily that
both permutations $u=(0\ 1)(4\ 5)(6\ 8)$ and  $v=(0\ 1\ 2)(4\ 3\ 5)$ are automorphisms of  $\widehat{\mathcal{A}_3}$. Since
$g(u)=(0\ 1)$ and $g(v)=(0\ 1\ 2)$, $g$ is surjective and 
$|\mathrm{Aut}(\widehat{\mathcal{A}_3})|=36$. \\
Let $K=\langle u,v\rangle$. Since $u^2=v^3=\mathrm{id}$ and $u\circ v=v^2\circ u$, we have
$K\simeq\mathfrak{S}_3$. Moreover 
$K\cap \ker g=\{\mathrm{id}\}$ thus finally  
$\mathrm{Aut}(\widehat{\mathcal{A}_3})=K\cdot \ker g$. Since $t\circ v\ne v\circ t$, this is a semi direct product, hence 
$\mathrm{Aut}(\widehat{\mathcal{A}_3})\simeq \mathfrak{S}_3\rtimes\mathfrak{S}_3$.

\item When $\Hy=\Hy'_3$ we have 
$$
\ker g=\langle q\rangle\simeq\mathfrak{S}_2
\quad \text{where\quad $q=(3\ 6)(4\ 7)(5\ 8)$.}
$$
Both permutations $r=(0\ 1)(3\ 4)(6\ 7)$ and  $s=(0\ 1\ 2)(4\ 3\ 5\ 7)$ are automorphisms of  $\Hy'_3$. Since
$g(r)=(0\ 1)$ and $g(s)=(0\ 1\ 2)$, $g$ is surjective and 
$|\mathrm{Aut}(\Hy'_3)|=12$. \\
Let $K'=\langle u',v'\rangle$ where $u'=(0\ 1\ 2)(4\ 8\ 3)(5\ 6\ 7)$ and $v'=(0\ 1)(3\ 4)(6\ 7)$ are automorphisms of $\Hy'_3$. Since ${u'}^3={v'}^2=\mathrm{id}$ and $v'\circ u'={u'}^2\circ v'$, we have
$K'\simeq\mathfrak{S}_3$. Moreover 
$K'\cap \ker g=\{\mathrm{id}\}$ thus finally  
$\mathrm{Aut}(\Hy'_3)=K'\cdot \ker g\simeq \mathfrak{S}_3\times\mathfrak{S}_2=D_6$ since clearly the product is direct.\qedhere
\end{itemize}
\end{proof}

\subsection{The general case}

\begin{thm}\label{theo44}
For any prime power $k\ge3$, there exists a hypergraph 
$\Hy'=\Hy'_k$ in $\mathrsfs{H}_k$ such that 
\begin{enumerate}

\item $\mathrm{q}(\Hy')=2k-1$ and any hyperedge $(2k-1)$-coloring of $\Hy'$
is unique up to an affine permutation on the vertices,

\item $\Hy'$ contains exactly $k$ many critical hyperedges,

\item $\Hy'$ is maximal (for inclusion) in the class of all linear and $k$-uniform
 hypergraphs with $k^2$ vertices,

\item $\Hy'$ and $\widehat{\mathcal{A}_k}$ are non isomorphic. 
\end{enumerate}
\end{thm}

\begin{proof}
In the field plane $\mathcal{A}_k$ let $x_0$ be an arbitrary vertex and $\Hy (x_0)=\{e_0,e_1,\dots,e_k\}$ be the star centered at $x_0$. Let $e_1=\{x_0,\dots,x_{k-1}\}$ and $a_1,\dots,a_{k-1}$ be the
parallel lines to $e_0$ with $a_i\cap e_1=\{x_i\}$, $1\le i\le k-1$. We apply to $\mathcal{A}_k$ the following changes on its set of hyperedges:
\begin{itemize}

\item $e_2,\dots,e_k$ are removed,

\item for any $1\le i\le k-1$, $a_i$ is replaced by the hyperedge $a'_i=\{x_0\}\cup(a_i\smallsetminus\{x_i\})$,

\item $e_1$ and all other hyperedges, namely those intersecting $e_0\smallsetminus\{x_0\}$ including $e_0$ itself, are kept unchanged.

\end{itemize}

\begin{center}
\begin{figure}[H]
\begin{tikzpicture}[scale=0.9]

\draw[thick] (1,0) -- (1,2);
\draw[thick,dashed] (1,2) -- (1,4);
\draw[thick] (1,4) -- (1,6);
\draw[thick] (2,0) -- (2,2);
\draw[thick,dashed] (2,2) -- (2,4);
\draw[thick] (2,4) -- (2,6);
\draw[thick] (6,0) -- (6,2);
\draw[thick,dashed] (6,2) -- (6,4);
\draw[thick] (6,4) -- (6,6);

\draw (0,-0.5) node {$e_0$};
\draw (1,-0.5) node {$a_1$};
\draw (2,-0.5) node {$a_2$};
\draw (6,-0.5) node {$a_{k-1}$};
\draw (6.5,6) node {$e_1$};
\draw (0,6) node {$\bullet$};
\draw (1,6) node {$\bullet$};
\draw (2,6) node {$\bullet$};
\draw (6,6) node {$\bullet$};
\draw (0,6.5) node {$x_0$};
\draw (1,6.5) node {$x_1$};
\draw (2,6.5) node {$x_2$};
\draw (6,6.5) node {$x_{k-1}$};

\draw [thick,dotted] (0,6) -- (6,0);
\draw [thick,dotted] (0,6) -- (3.25,-0.5);
\draw [thick,dotted] (0,6) -- (6.5,2.75);
\draw (3.2,4.8) node {$e_{u-1}$};
\draw (3,3.4) node {$e_u$};
\draw (3,1.2) node {$e_{u+1}$};

\draw (0,5) node {$\bullet$};
\draw (0,4) node {$\bullet$};
\draw (0,2) node {$\bullet$};
\draw (0,1) node {$\bullet$};
\draw (0,0) node {$\bullet$};

\draw (1,5) node {$\bullet$};
\draw (1,4) node {$\bullet$};
\draw (1,2) node {$\bullet$};
\draw (1,1) node {$\bullet$};
\draw (1,0) node {$\bullet$};

\draw (2,5) node {$\bullet$};
\draw (2,4) node {$\bullet$};
\draw (2,2) node {$\bullet$};
\draw (2,1) node {$\bullet$};
\draw (2,0) node {$\bullet$};

\draw (6,5) node {$\bullet$};
\draw (6,4) node {$\bullet$};
\draw (6,2) node {$\bullet$};
\draw (6,1) node {$\bullet$};
\draw (6,0) node {$\bullet$};

\draw[thick,rounded corners] (-0.2,5.8) rectangle (6.2,6.2);
\draw[thick,rounded corners] (-0.2,-0.2) rectangle (0.2,6.2);

\draw[dotted] (0,0) -- (6,0);
\draw[dotted] (0,1) -- (6,1);
\draw[dotted] (0,2) -- (6,2);
\draw[dotted] (0,4) -- (6,4);
\draw[dotted] (0,5) -- (6,5);

\draw (3,7.2) node {$\mathcal{A}_k$};
\end{tikzpicture}\quad
\begin{tikzpicture}[scale=0.9]
\draw[thick] (1,0) -- (1,2);
\draw[thick,dashed] (1,2) -- (1,4);
\draw[thick] (1,4) -- (1,5);
\draw[thick] (1,5) -- (0,6);
\draw[thick] (2,0) -- (2,2);
\draw[thick,dashed] (2,2) -- (2,4);
\draw[thick] (2,4) -- (2,5);
\draw[thick] (2,5) -- (0,6);
\draw[thick] (6,0) -- (6,2);
\draw[thick,dashed] (6,2) -- (6,4);
\draw[thick] (6,4) -- (6,5);
\draw[thick] (6,5) -- (0,6);

\draw (0,-0.5) node {$e_0$};
\draw (1,-0.5) node {$a'_1$};
\draw (2,-0.5) node {$a'_2$};
\draw (6,-0.5) node {$a'_{k-1}$};
\draw (6.5,6) node {$e_1$};
\draw (6.5,5) node {$e'_2$};
\draw (6.5,4) node {$e'_3$};
\draw (6.7,2) node {$e'_{k-3}$};
\draw (6.7,1) node {$e'_{k-2}$};
\draw (6.7,0) node {$e'_{k-1}$};
\draw (-0.5,5) node {$y_1$};
\draw (-0.5,4) node {$y_2$};
\draw (-0.7,2) node {$y_{k-3}$};
\draw (-0.7,1) node {$y_{k-2}$};
\draw (-0.7,0) node {$y_{k-1}$};
\draw (0,6) node {$\bullet$};
\draw (1,6) node {$\bullet$};
\draw (2,6) node {$\bullet$};
\draw (6,6) node {$\bullet$};
\draw (0,6.5) node {$x_0$};
\draw (1,6.5) node {$x_1$};
\draw (2,6.5) node {$x_2$};
\draw (6,6.5) node {$x_{k-1}$};


\draw (0,5) node {$\bullet$};
\draw (0,4) node {$\bullet$};
\draw (0,2) node {$\bullet$};
\draw (0,1) node {$\bullet$};
\draw (0,0) node {$\bullet$};

\draw (1,5) node {$\bullet$};
\draw (1,4) node {$\bullet$};
\draw (1,2) node {$\bullet$};
\draw (1,1) node {$\bullet$};
\draw (1,0) node {$\bullet$};

\draw (2,5) node {$\bullet$};
\draw (2,4) node {$\bullet$};
\draw (2,2) node {$\bullet$};
\draw (2,1) node {$\bullet$};
\draw (2,0) node {$\bullet$};

\draw (6,5) node {$\bullet$};
\draw (6,4) node {$\bullet$};
\draw (6,2) node {$\bullet$};
\draw (6,1) node {$\bullet$};
\draw (6,0) node {$\bullet$};

\draw[thick,rounded corners] (-0.2,5.8) rectangle (6.2,6.2);
\draw[thick,rounded corners] (-0.2,-0.2) rectangle (0.2,6.2);


\draw[dotted] (0,0) -- (6,0);
\draw[dotted] (0,1) -- (6,1);
\draw[dotted] (0,2) -- (6,2);
\draw[dotted] (0,4) -- (6,4);
\draw[dotted] (0,5) -- (6,5);

\draw (3,7.2) node {$\Hy'$};

\end{tikzpicture}
\caption{From $\mathcal{A}_k$ to $\Hy'$. Vertices of $e_0$ and $e_1$ are labelled so that the lines $\langle x_i,y_i\rangle$ 
($1\le i\le k-1$) are parallel in $\mathcal{A}_k$ and $\Hy'$.}
\label{fig4}
\end{figure}
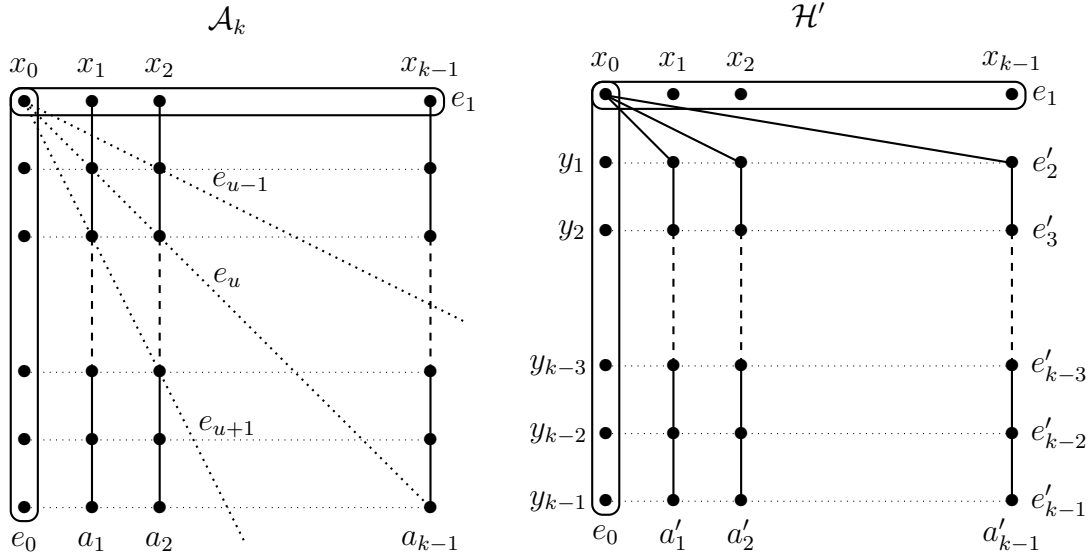
\end{center}

We get a $k$-uniform hypergraph $\Hy'=(V',E')$ such that $|V'|=|V(\mathcal{A}_k)|=k^2$, $|E'|=|E(\mathcal{A}_k)|-(k-1)=k^2+1$  and $\mathrm{deg}_{\Hy'}(x)=\mathrm{deg}_{\mathcal{A}_k}(x)=k+1$ for any 
$x\in e_0$. Moreover $V'\smallsetminus e_0=\{x_{i,j},\ 1\le i\le k,\ 
1\le j\le k-1\}$ where $\{x_{i,j}\}=e_i\cap a_j$. It  follows that
$\mathrm{deg}_{\Hy'}(x_{i,j})=\mathrm{deg}_{\mathcal{A}_k}(x_{i,j})-1$. Finally $\Hy'\in \mathrsfs{H}_k$.

\medskip
\noindent
i)
We now estimate $\mathrm{q}(\Hy')$.  

\begin{itemize}
\item We first fix a hyperedge coloring 
$\mathfrak{c}:\Hy'\longrightarrow\mathbb{N}$ of $\Hy'$. 
Hyperedges in the star $\Hy'(x_0)=\{e_0,a'_1,\dots,a'_{k-1},e_1\}$ are 
respectively colored by  $k+1$ many distinct colors in the set 
$\mathcal{C}=\{c_0,c_1,\dots,c_{k-1},c_k\}$. Now
the star $\Hy'(x_1)$ is composed by $e_1$ and the lines $\langle x_1,y_i\rangle$  
$(1\le i\le k-1$) of $\mathcal{A}_k$ passing thru $x_1$ and $y_i$ respectively 
where we have set $e_0=\{x_0,y_1,\dots,y_{k-1}\}$. 
Vertices of $e_0$ are labelled so that the hyperedges passing thru $x_i,y_i$ ($1\le i\le k-1$) are parallel. We denote by $c'_i$ 
($1\le i\le k-1$) the $k-1$ many distinct colors that are assigned to those lines $\langle x_1,y_i\rangle$ ($1\le i\le k-1$). 
For any $2\le j\le k-1$, each line $\langle x_1,y_i\rangle$ intersects $a_j$ in 
$\mathcal{A}_k$, thus intersects $a'_j$ in $\Hy'$ since the intersection vertex 
cannot be $x_j$ by linearity of $\Hy'$. Thus $c'_i\not\in \mathcal{C}\smallsetminus\{c_1\}$ for any $1\le i\le k-1$. Since there is at most 
one line $\langle x_1,y_i\rangle$ ($1\le i\le k-1$) that are colored by $c_1$,
we infer $\mathrm{q}(\Hy')\ge k+1+k-2=2k-1$. 
\\
Note that if none of the lines 
$\langle x_1,y_i\rangle$ ($1\le i\le k-1$) is colored by $c_1$, then 
the given hyperedge coloring $\mathfrak{c}$ has at least $2k$ many colors. 
We use this fact in order to get an optimal coloring of $\Hy'$ in what follows.

\item 
We proceed as follows. We color as before $\Hy'(x_0)$ by $k+1$ many distinct colors $c_0,\dots,c_k$. We observed that $k-2$ among the $k-1$ lines 
$\langle x_1,y_i\rangle$ $(1\le i\le k-1)$ require new colors
$c_{k+1},\dots,c_{2k-2}$. We have to show that we can color the remaining
hyperedges of $\Hy'$ without any additional color.\\[0.2em]
Let $e$ be a not yet colored hyperedge of $\Hy'$. \\
-- if $e$ is the last line $\langle x_1,y_{i_1}\rangle$ passing thru $x_1$, then
we must color it with $c_1$; we denote by $\lambda\in\mathbb{F}_k\smallsetminus\{0\}$ its slope in the field plane $\mathcal{A}_k$;\\
-- if $e\in\{e'_2,\dots,e'_{k-1}\}$ is a parallel line to $e_1$, then we must color it with $c_k$ since it intersects $e_0,a'_i,\langle x_1,y_{i}\rangle$ ($1\le i\le k-1$);\\
-- for each $2\le j\le k-1$, the unique parallel $e=\langle x_j,y_{i_j}\rangle$ to 
$\langle x_1,y_{i_1}\rangle$ needs to be colored with $c_j$ since it intersects
$e_0,e_1,a'_i$ $(i\ne j)$ and $\langle x_1,y_i\rangle$ $(i\ne i_1)$;
\\
--  for any $2\le j\le k-1$ and $1\le i\le k-1$, $i\ne i_j$,
the line $e=\langle x_j,y_{i}\rangle$ can only be colored with the same color than that of its unique parallel in $\Hy'(x_1)$.\\
By this way only hyperedges of $\Hy'$ that are parallel lines in $\mathcal{A}_k$ may have the same color. Indeed by construction, each color $c_{k+i-1}$ ($1\le i\le k-1$) is only assigned to disjoint hyperedges of $\Hy'$ which are parallel lines in
$\mathcal{A}_k$. Color $c_0$ is exclusive to $e_0$. Finally color $c_j$ ($1\le j\le k-1$) is only assigned to both hyperedges $a'_j$ and $\langle x_j,y_{i_j}\rangle$ which are non intersecting. \\  
We thus have a hyperedge 
$(2k-1)$-coloring of $\Hy'$. Hence $\mathrm{q}(\Hy')=2k-1$.

\item We have shown in fact that hyperedge $(2k-1)$-coloring of $\Hy'$ are parametrized by 
the slope  $\lambda\in\mathbb{F}_k\smallsetminus\{0\}$ of the parallel lines
$\langle x_j,y_{i_j}\rangle$ ($1\le j\le k-1$). Conversely different slopes $\lambda$ yield distinct coloring when vertices and hyperedges of $\Hy'$ are fixed beforehand. Hence if we ignore color labels we infer that $\Hy'$ admits exactly 
$k-1$ many hyperedge $(2k-1)$-coloring $\mathfrak{c}_{\Hy',\lambda}$,
$\lambda\in\mathbb{F}_k\smallsetminus\{0\}$.\\
Precisely each slope $\lambda\in\mathbb{F}_k\smallsetminus\{0\}$ is associated to some admissible one to one 
mapping $j\longmapsto i_j$ from $\{1,\dots,k-1\}$ on itself. It
yields a permutation $\sigma$ on the $y_j$'s and the $e'_j$'s respectively  
$\sigma:y_j\longmapsto y_{i_j}$, $e'_j\longmapsto e'_{i_j}$ which extends to a one to one mapping from the vertices of $e'_j$ onto
those of $\sigma(e'_j)$ ($1\le j\le k-1$) by sending $a'_i\cap e'_j$ on
$a'_i\cap \sigma(e'_j)$ ($1\le i\le k-1$), see Figure \ref{fig4}. 
Since $\Hy'$ is immersed in the field plane 
$\mathcal{A}_k=\mathbb{F}_k^2$ we may consider
that $e_0=\{0\}\times \mathbb{F}_k$ and $e_1=\mathbb{F}_k\times\{0\}$.
It follows $\sigma$ is an affine transformation of the field plane, namely 
the \emph{dilation from axis $e_1$ towards axis $e_0$ by a factor of $\lambda$}. We have $\sigma(\Hy')=\Hy'$. Indeed\\
-- $\sigma$ permutes the $e'_j$'s,\\
-- $\sigma(e_0)=e_0$, $\sigma(e_1)=e_1$, $\sigma(a'_i)=a'_i$ $(2\le i\le k-1)$,\\
-- $\sigma(\langle x_i,y_j\rangle)=\langle x_i,\sigma(y_j)\rangle$  ($1\le i,j\le k-1$).\\
It is thus an automorphim of $\Hy'$.  Finally 
$$
\mathfrak{c}_{\Hy',\lambda}(e)
=\mathfrak{c}_{\Hy',1}(\sigma^{-1}(e)),\quad \text{for any $e\in E(\Hy')$.}
$$
This shows that for any coloring $\mathfrak{c}$ of $\Hy'$ there exists an \emph{affine} automorphism of $\Hy'$ such that $\mathfrak{c}(e)=
\mathfrak{c}_{\Hy',1}(\sigma^{-1}(e))$ for any $e\in E(\Hy')$.
\end{itemize}

\noindent
ii) Here we show that only 
$e_0,a'_1,\dots,a'_{k-1}$ are critical hyperedges in $\Hy'$.

\begin{itemize}
\item By the coloring process implemented above, $a'_j$ ($2\le j\le k-1$) is the unique hyperedge of $\Hy'$ with color $c_j$. By symmetry
on the vertices $x_i$ ($2\le i\le k-1$),  there exists a hyperedge coloring of $\Hy'$ such that $a'_1$ has an exclusive color $c_{1}$.  Hence any $a'_j$ ($1\le j\le k$) is critical. It  is clear that $e_0$ is also critical since it intersects all other 
hyperedges.

\item
Any parallel of $e_1$, including $e_1$ itself, intersects all the hyperdeges of $\Hy'$
which are not parallel to $e_1$. Removing one of them does not affect
the number of required colors. Thus none of them is critical.

\item In the first part of the proof of i), we
notice that just coloring $e_0,e_1,a'_i,\langle x_1,y_i\rangle$ ($1\le i\le k-1$) already required $2k-1$ colors.  We infer that none of the lines $\langle x_j,y_i\rangle$ is critical, $2\le j\le k-1$ and $1\le i\le k-1$. By symmetry of $\Hy'$ with regard to 
the vertices $x_j$, this is also true for the lines $\langle x_1,y_i\rangle$,
$1\le i\le k-1$.

\end{itemize}

\noindent
iii) 
Let $\Hy''$ be a $k$-uniform linear hypergraph with $k^2$ vertices such that 
$E(\Hy')\subset E(\Hy'')$. Let $e\in E(\Hy'')$.  Set $e'_1=e_1$ and denote by  $e'_2,\dots,e'_{k}$ its parallel lines in $\mathcal{A}_k$ (see Figure~\ref{fig4}).
These hyperedges are also in $\Hy'$ hence form a partition 
of $V(\Hy')$. We thus have $e\cap e'_i=\{z_i\}$ 
for distinct vertices $z_i$, $1\le i\le k$, and
$e=\{z_1,\dots,z_k\}$.
\begin{itemize}

\item If $z_1=x_0$ then $z_2$ is not in $e'_1=e_1$ and does belong to some 
$a'_j$. Since $x_0\in a'_j\in E(\Hy')\subset E(\Hy'')$ and $\Hy''$ is linear, we infer $e=a'_j\in E(\Hy')$.

\item Assume that $z_1=x_j$ for some $1\le j\le k-1$. Since $k\ge3$ and
$|a'_j\cap e|\le1$, there
exists $2\le i\le k$ such that $z_i\not\in a'_j$. This implies that 
$z_1$ and $z_i$ lie on two distinct proper parallels to $e_0$ in 
$\mathcal{A}_k$. We then see that the line $\langle z_1,z_i\rangle$ containing both $z_1$ and $z_i$
in  $\mathcal{A}_k$
is distinct from the $a_t$'s, $1\le t\le k-1$, and from the $e_t$'s, $0\le t\le k$. Thus it is a hyperedge of $\Hy'$ as well of $\Hy''$. By linearity of $\Hy''$
we conclude that $e=\langle z_1,z_i\rangle\in E(\Hy')$.
\end{itemize}

\noindent
iv) In view of Proposition \ref{pp42} the result follows equivalently from i), ii)  or iii).
\end{proof}

\subsection{Remarks and questions}
\begin{enumerate}

\item As a generalization of iii) of Theorem~\ref{theo43} and an extension of iv) of Theorem~\ref{theo44}, it should be interesting to determine and to compare the groups of automorphisms of $\widehat{\mathcal{A}_k}$ and 
$\Hy'_k$.

\item In Proposition \ref{pp42} we investigated the case when $k$ is a prime power and we proved that $\mathrsfs{H}_k$ is not empty
and that  $|\mathrsfs{H}_k|\ge2$. Estimating $|\mathrsfs{H}_k|$ seems to be a hard problem.

\item When $k$ is not a prime power it is not obvious neither known that $\mathrsfs{H}_k$ is non-empty. In contrast with the prime power case we should mention that it is conjectured (known as the prime power Conjecture) that there is no affine plane with $k^2$ points. For instance for $k=6$ this is related to the $36$ officers problem posed by
Euler and to the existence of Graeco-Latin squares of order~$6$: there is no affine plane of order~$6$ (see \cite{Tarry1900,Stinson1984}). 
However one can ask the question whether or not $\mathrsfs{H}_6$ is empty.

\item 
Our construction of $\Hy'_k$ the Theorem \ref{theo44} starts from the affine
plane $\mathcal{A}_k$. We may ask the question whether or not all hypergraphs of  $\mathrsfs{H}_k$ can be derived in a certain way from 
$\mathcal{A}_k$. However,
as also noticed in \cite{Dow1986}, we need to avoid hypergraphs in which 
parallel is an equivalence relation.
\end{enumerate}

\section{\bf Further results}

\subsection{Some consequences}

We first derive the following result establishing Conjecture~\ref{conj14} for hypergraphs with small maximum degree.

\begin{thm}\label{theo51}
Let $\Hy =(V,E)$ be a linear hypergraph with $n$ vertices and without loop such 
that $\Delta(\Hy)\le\sqrt{n}+1$. Then $\mathrm{q}(\Hy)\le n$. 
\end{thm}

\begin{proof}
Let $\Hy'=(V,E')$ be a critical partial hypergraph of $\Hy$ such that $\mathrm{q}(\Hy')=\mathrm{q}(\Hy)$. 
\begin{itemize}
\item If $\mathrm{ar}(\Hy')\ge\sqrt{n}$ we infer from \cite{Sanchez2008} or 
from Theorem \ref{theo16} that $\mathrm{q}(\Hy)=\mathrm{q}(\Hy')\le n$.

\item 
If $\mathrm{ar}(\Hy')<\sqrt{n}$ then we take $e\in E'$ such that
$|e|=\mathrm{ar}(\Hy')$.  By Lemma \ref{lem32} and~\eqref{eqidenhyp2} we get
\begin{align*}
\mathrm{q}(\Hy)-1=\mathrm{q}(\Hy')-1&\le \mathrm{d}_{\Hy'}(e)
=\sum_{x\in e}(\mathrm{deg}_{\Hy'}(x)-1)\\
&\le |e|(\Delta(\Hy')-1)\le 
\mathrm{ar}(\Hy')(\Delta(\Hy)-1)<n.
\end{align*}

\end{itemize}
 Hence the result.
\end{proof}

Note that by \eqref{lineg} we have $\mathrm{q}(\Hy)\le n$ whenever
$\mathrm{r}(\Hy)(\Delta(\Hy -1)<n$. Thus by letting 
$u=\Delta(\Hy)-\sqrt{n}-1$ with $0<u\le \sqrt{n}-2$  we see that  Conjecture~\ref{conj14} holds under the additional strong hypothesis on the rank $\mathrm{r}(\Hy)\le \sqrt{n}-u$.

\smallskip
The second statement extends the validity of Theorem \ref{theo16} when $\Hy$ is a uniform linear hypergraph. The extension is moderate but allows us to drop down the threshold $\sqrt{|V|}$ for the antirank.

\begin{thm}\label{theo52}
Let $k\ge2$ and $\Hy =(V,E)$ be a $k$-uniform linear hypergraph with $n$ vertices satisfying  $n\le k^2+k-2$. Then $\mathrm{q}(\Hy)\le \Delta([\Hy]_2)+1$.
\end{thm}

\begin{proof}
By Theorem \ref{theo16} we may assume that $n\ge k^2$. We set $u=n-k^2$. Hence
\begin{equation}\label{equ=}
0\le u\le k-2.
\end{equation}
Since any partial hypergraph of $\Hy$ satisfies theorem's hypotheses, we may assume that $\Hy$ is critical. 
 We distinguish two cases

\begin{itemize}
\item We first assume that there exists $e_0\in E$ and $x_0\in e_0$ such that 
$\mathrm{deg}_{\Hy}(x_0)\le |e_0|=k$. By Lemma \ref{lem32} and \eqref{eqidenhyp2} we get
\begin{align*}
\mathrm{q}(\Hy)-1\le \mathrm{d}_\Hy (e_0)&=\sum_{x\in e_0}(\mathrm{deg}_{\Hy}(x)-1)\\
&=\sum_{x\in e_0\smallsetminus\{x_0\}}\mathrm{deg}_{\Hy}(x)-|e_0|+\mathrm{deg}_{\Hy}(x_0)\\
&\le (|e_0|-1)\Delta(\Hy)= (k-1)\Delta(\Hy)= 
\Delta([\Hy]_2).
\end{align*}

\item We now assume that $\forall x\in V$ $\mathrm{deg}_{\Hy}(x)\ge k+1$. By \eqref{eqdex} we get
$$
\forall x\in V\quad \mathrm{deg}_{\Hy}(x)\le \frac{k^2+k-3}{k-1}<k+2.
$$
hence $\Hy$ is $(k+1)$-regular. 
By \eqref{eqiden} we thus infer
$$
|E|=\frac{(k+1)n}{k}=k^2+k+u+\frac uk,
$$
implying $u=0$ by \eqref{equ=} since $|E|$ is an integer. Whence
$|V|=k^2$ and $|E|=k^2+k$. Our aim is to show that $\Hy$ is a finite affine plane yielding $\mathrm{q}(\Hy)=k+1$.

For any $x\in V$ we have $\mathrm{deg}_{[\Hy]_2}(x)=\sum_{e\in\Hy (x)}(|e|-1)=k^2-1=|V|-1$ hence any pair of distinct vertices $(x,y)$ in $\Hy$ are adjacent. Denote by $e_{x,y}$ the unique hyperedge containing both $x$ and $y$.

Let $e\in E$. Then by \eqref{eqidenhyp2} we have $\mathrm{d}_\Hy (e)=\sum_{x\in e}(\mathrm{deg}_{\Hy}(x)-1)=k^2$.  Since $|E|=k^2+k$ and 
$\mathrm{d}_\Hy (e)=k^2$, there are $k-1$ many hyperedges $a_1,\dots,a_{k-1}$ such that
$a_i\cap e=\varnothing$, $1\le i\le k-1$. If there were $y\in a_i\cap a_j$ with $i\ne j$,  then
$\Hy (y)$ would contain $k+2$ distinct hyperedges: $a_i,a_j$, and $e_{x,y}$, $x\in e$, a contradiction to $|\Hy (y)|=\mathrm{deg}_{\Hy}(y)=k+1$. Hence $a_i\cap a_j=\varnothing$ if $i\ne j$. This gives
$$
\left|e\cup \bigcup_{i=1}^{k-1}a_i\right|=|e|+\sum_{i=1}^{k-1}|a_i|=k^2=|V|,
$$
thus any $x\in V\smallsetminus e$ belongs to exactly one of the hyperedges $a_i$, $1\le i\le k-1$, that is a parallel to $e$. We conclude that $\Hy$ is a finite affine plane and that $\mathrm{q}(\Hy)=k+1$.\qedhere
\end{itemize}
\end{proof}

\subsection{A generalization of Theorem \ref{theo16}}
Let $\Hy =(V,E)$ be a hypergraph with $n$ vertices. Then we can split $E$ as the disjoint union $E=E'\sqcup E''_{-}\sqcup
E''_{+}$ where
\begin{align*}
E'&=\{e\in E: |e|\ge \sqrt{n}\},
\\
E''_{-}&=\{e\in E\smallsetminus E':
\exists x_0\in e\text{ such that }\mathrm{deg}_{\Hy}(x_0)\le |e|\},\\
 E''_{+}&=\{e\in E\smallsetminus E':
\forall x\in e, \ \mathrm{deg}_{\Hy}(x) > |e|\}.
\end{align*}
Theorems \ref{theo15} and \ref{theo16} can be combined in order to obtain the 
following slight extension of both results. It shows that the bound \eqref{eqqH} is true
when $E''_{+}=\varnothing$.

\begin{thm} \label{theo54}
Let $\Hy =(V,E)$ be a linear hypergraph with $n$ vertices and without loop such that
$E$ can be partitioned as $E=E'\sqcup E''$ in such a way that the partial hypergraphs $\mathcal{H'}=(V, E')$ and $\mathcal{H''}=(V, E'')$ satisfy the following properties:
\begin{enumerate}
\item  $\mathrm{ar}(\mathcal{H'})\geq \sqrt{n}$,

\item $\forall e\in  E''$, $\vert e\vert< \sqrt{n}$,

\item  $\forall e\in E''$, $\exists x_0\in e$ such that $\mathrm{deg}_{\Hy}(x_0)\leq \vert e\vert$.

\end{enumerate}
Then
$\mathrm{q}(\Hy)\le \Delta([\Hy]_2)+1.$
\end{thm}
\begin{proof}
We argue by induction on $\vert  E''\vert $. It runs as in the proof of Theorem \ref{theo15}.
\\[0.2em]
When $E''=\varnothing$ then $\Hy =\Hy'$ and the result follows directly from Theorem \ref{theo16}.
\\[0.2em]
When $E''=\{e_0,e_1,\dots,e_m\}$ then
the partial hypergraph $$\Hy _0:=\Hy \smallsetminus e_0=(V,E'\sqcup(E''\smallsetminus\{e_0\}))$$ satisfies 
conditions i)-iii) and the induction hypothesis applies to $\Hy _0$. We even can assume that $|e_0|=\mathrm{ar}(\Hy)$ since $|e_i|<\sqrt{n}\le\mathrm{ar}(\Hy')$ for all $0\le i\le m$.
We distinguish two cases.
\begin{itemize}
\item If $\mathrm{q}(\Hy)=\mathrm{q}(\Hy _0)$, then by applying induction hypothesis
to $\Hy _0$ we infer 
$$
\mathrm{q}(\Hy)=\mathrm{q}(\Hy _0)\le \Delta([\Hy _0]_2)+1\le \Delta([\Hy]_2)+1.
$$
\item If $\mathrm{q}(\Hy)=\mathrm{q}(\Hy _0)+1$ then we fix $x_0\in e_0$ 
such that $\mathrm{deg}_{\Hy}(x_0)\le |e_0|$
and we get by Lemma \ref{lem32} 
\begin{align*}
\mathrm{q}(\Hy)&\le \mathrm{d}_\Hy (e_0)+1=\sum_{\substack{x\in e_0\\ x\ne x_0}}(\mathrm{deg}_{\Hy}(x)-1)
+\mathrm{deg}_{\Hy}(x_0)\\
&\le (|e_0|-1)\max_{x\in e}\mathrm{deg}_{\Hy}(x)-|e_0|+1+\mathrm{deg}_{\Hy}(x_0)\\
&\le \Delta([\Hy]_2)+1,
\end{align*}
since $|e_0|=\mathrm{ar}(\Hy)$.
\end{itemize}
This completes the induction.
\end{proof}

\begin{rems}\label{rm55}
\begin{enumerate}
\item Before applying this theorem it should be convenient to remove all isolated vertices
and all but one vertices of degree one in each hyperedge 
so that the resulting hypergraph still verifies conditions i)-iii).
Indeed this will not change the chromatic index but may improve on
the upper bound.

\item
Let $\Hy =(V,E)$ and assume that $|E''_{+}|=r\ge1$.
The partial hypergraph $\tilde{\Hy}=(V,E\smallsetminus E''_{+})$ satisfies
Theorem \ref{theo54}. Letting $\Hy''_+=(V,E''_{+})$
we infer
\begin{equation}\label{eqremfin}
\mathrm{q}(\Hy)\le \mathrm{q}(\tilde{\Hy})+\mathrm{q}(\Hy''_+)\leq
\Delta([\tilde{\Hy}]_2)+1+r\leq
\Delta([\Hy]_2)+1+r.
\end{equation}
Note that $r$ can be very large, as large as the total number of hyperedges.
If $\Hy =(V,E)$ is a $k$-uniform and $(k+1)$-regular linear hypergraph with $k^2$ vertices,
then $|E|=k(k+1)$ by \eqref{eqiden}. Adding an isolated  vertex to $\Hy$
gives a hypergraph $\hat{\Hy}=(\hat{V},\hat{E})$ in which
$\hat{E}=\hat{E}''_{+}$.   In this case the bound \eqref{eqremfin} is trivial.

\item In fact the induction argument in the proof of Theorem \ref{theo54} 
applies as well on the set of hyperedges 
 $$
E_{-}:=\{e\in E\,:\, \exists x_0\in e \text{ such that }
\mathrm{deg}_{\Hy}(x_0)\le |e|\}.
$$
Hence if the partial hypergraph $(V,E\smallsetminus E_{-})$ satisfies Berge-F\"uredi bound
\eqref{eqqH}
then $(V,E)$ also does. It follows that Conjecture \ref{conj12} is true in general if it holds for every hypergraph $(V,E)$ such that $E_{-}$ is empty.
\end{enumerate}
\end{rems}

\subsection{Consequences for the {Erd\H{o}s-Faber-Lov\'asz} Conjecture}
We noticed earlier that  
Conjecture \ref{conj12} implies Conjecture \ref{conj14}, since $\Delta([\Hy]_{2})\leq |V|-1$ for any hypergraph $\Hy =(V,E)$. Therefore we have the following corollaries.
\begin{cor}\label{cor56}
Let $\Hy$ be a linear hypergraph with $n$ vertices and satisfying \eqref{cond*}.
Then 
$\mathrm{q}(\Hy)\le n$.
\end{cor}

\begin{cor}\label{cor57}
Let $\Hy$ be a linear hypergraph with $n$ vertices and satisfying
$\mathrm{ar}(\Hy)\ge n^{1/2}$. 
Then 
$\mathrm{q}(\Hy)\le n$.
\end{cor}

\begin{cor}\label{cor58}
Let $\Hy =(V,E'\sqcup E'')$ be a linear hypergraph with $n$ vertices  and satisfying  hypotheses of Theorem \ref{theo54}. Then $\mathrm{q}(\Hy)\le n$.
\end{cor}

\end{document}